\documentclass[twocolumn,letterpaper]{IEEEtran} 

\IEEEoverridecommandlockouts                              

\usepackage{cite}
\usepackage{amsfonts}
\usepackage{amssymb}
\usepackage{amsmath}
\usepackage{graphicx}

\usepackage{stdmath_noieee}
\usepackage{qi}

\graphicspath{{./graphics_included/}}

\def\lcplus{\textbf{ edited}}

\newcounter{edremcounter}
\def\torig{\tilde{t}}
\def\porig{\tilde{p}}
\def\tvar{t}

\def\H{H}
\def\h{h}
\def\NN{\mathcal{N}}

\def\onetom{1,\ldots ,m}
\def\oneton{1,\ldots ,n}
\def\onetony{1,\ldots ,n_y}
\def\onetonu{1,\ldots ,n_u}

\def\vecc{\text{vec}}

\def\B{\mathbb{B}}

\def\bigdelay{R}
\def\m{m} 

\def\kgkbin{\bin{K}(\bin{G}\bin{K})}

\begin{document}


\title{\LARGE \bf
On the Nearest Quadratically Invariant Information Constraint}

\author{Michael C. Rotkowitz \quad and \quad Nuno C. Martins
\thanks{M.C. Rotkowitz is with the Department of Electrical and
  Electronic Engineering,
  The University of Melbourne, Parkville VIC 3010 Australia,
  {\tt\small mcrotk@unimelb.edu.au}}%
\thanks{N.C. Martins is with the Department of Electrical and Computer
  Engineering and the Institute for Systems Research,
   The University of Maryland, College Park MD 20740 USA,
  {\tt\small nmartins@umd.edu}}%
}


\maketitle



\begin{abstract}
Quadratic invariance is a condition which has been shown to allow for
  optimal decentralized control problems to be cast as convex
  optimization problems.  The condition relates the constraints that the
  decentralization imposes on the controller to the structure of the
  plant.
In this paper, we consider the problem of
  finding the closest subset and superset of the decentralization
  constraint which are quadratically invariant when the original problem is not.
We show that this can itself be cast as a convex problem for the
  case where the controller is subject to delay constraints between
  subsystems, but that this fails when we only consider sparsity
  constraints on the controller.  For that case, we develop an
  algorithm that finds the closest superset in a fixed number of steps,
  and discuss methods of finding a close subset.



\end{abstract}



\section{Introduction}
\label{sec:intro}



The design of decentralized controllers has been of interest for a
long time, as evidenced in the surveys~\cite{witsenhausen_1971,sandell_1978}, and
continues to this day with the advent of complex interconnected
systems.
 The counterexample  constructed by Hans Witsenhausen in 1968~\cite{witsenhausen_1968} clearly illustrates the fundamental
reasons why problems in decentralized control are difficult.

Among the recent results in decentralized control, new approaches have been introduced that are based
on algebraic principles, such as the work in~\cite{rotkowitz_lall_tac05,qi_note,voulgaris_2001}.
Very relevant to this paper is the work
in~\cite{rotkowitz_lall_tac05,qi_note}, which classified the
problems for which optimal decentralized synthesis
could be cast as a convex optimization problem.
Here, the plant is linear, time-invariant and it
is partitioned into dynamically coupled subsystems, while the controller is also partitioned into subcontrollers.
In this framework, the decentralization being imposed
manifests itself as constraints on the controller to be designed, often
called the \emph{information constraint}.

The information constraint on the overall controller specifies what information is
available to which controller.
For instance, if information is passed between subsystems, such that
each controller can access the outputs from other subsystems after
different amounts of transmission time, then the information constraints are delay constraints,
and may be represented by a matrix of these transmission delays.
If instead, we consider each controller to be able to access the
outputs from some subsystems but not from others, then the information
constraint is a sparsity constraint, and may be represented by a
binary matrix.

Given such pre-selected information constraints, the existence of a
convex parameterization for all stabilizing controllers that satisfy the constraint can be determined via the algebraic test
introduced in~\cite{rotkowitz_lall_tac05,qi_note}, which is denoted as \emph{quadratic invariance}.
In contrast with prior work, where the information constraint on the controller is fixed beforehand,
this paper addresses the design of the information constraint itself. More specifically,
given a plant and a pre-selected information constraint that is not quadratically invariant, we give
explicit algorithms to compute the quadratically invariant information constraint that is closest to the pre-selected one.
We consider finding the closest quadratically invariant superset,
which corresponds to relaxing the pre-selected constraints as little as
possible to get a tractable decentralized control problem, which may
then be used to obtain a lower bound on the original problem, as well
as finding the closest quadratically invariant subset,
which corresponds to tightening the pre-selected constraints as little as
possible to get a tractable decentralized control problem, which may
then be used to obtain upper bounds on the original problem.

We consider the two particular cases of information constraint outlined above.
 In the first case, we consider constraints as transmission delays between
the output of each subsystem and the subcontrollers that are connected
 to it. The distance between any two information constraints
is quantified via a norm of the difference between the delay
matrices, and we show that we can find the closest quadratically
invariant set, superset, or subset as a convex optimization
problem.

In the second case, we consider sparsity constraints that represent
which controllers can access which subsystem outputs, and represent
such constraints with binary matrices.
The distance between information constraints is then given by the hamming distance,
applied to the binary sparsity matrices.
We provide an algorithm that gives the closest superset; that is,
the quadratically invariant constraint that can be obtained 
 by way of \emph{allowing} the least number of additional links, and
 show that it terminates in a fixed number of iterations.
For the problem of finding a close set or subset, we discuss some heuristic-based solutions.

\textbf{Paper organization:} Besides the introduction, this paper has six sections. Section~\ref{sec:prelims} presents
the notation and the basic concepts used throughout the paper. The delay and sparsity constraints adopted in
our work are  described in detail in Section~\ref{sec:probform},
while their characterization using quadratic invariance is given in Section~\ref{sec:qi}. The
main problems addressed in this paper are formulated  and solved  in Section~\ref{sec:closest}.
Section~\ref{sec:nltv} briefly notes how this work also applies when
assumptions of linear time-invariance are dropped, numerical examples
are given in Section~\ref{sec:numex}, and conclusions are given in Section~\ref{sec:conc}.

\section{Preliminaries}
\label{sec:prelims}



Throughout the paper, we adopt a given causal linear time-invariant continuous-time plant $P$ partitioned as follows:
\[
P=\bmat{P_{11} & P_{12} \\ P_{21} & G}
\] Here, $P\in\rp^{(n_y+n_z) \times (n_w+n_u)}$, where
$\rp^{q \times r}$ denotes the set of matrices of dimension $q$ by $r$,
whose entries are proper transfer functions of the Laplace complex variable $s$. 
Note that we abbreviate $G=P_{22}$, since we will
refer to that block frequently, and so that we may refer to its
  subdivisions without ambiguity. 

 Given a causal linear time-invariant controller $K$ in $\rp^{n_u\times n_y}$, we define the \eemph{closed-loop map} by
\[
f(P,K) \overset{def}{=} P_{11} + P_{12} K (I - GK)^{-1}P_{21}
\] where we assume that the feedback interconnection is well posed.
The map $f(P,K)$ is also called the (lower) \eemph{linear fractional
  transformation} (LFT) of $P$ and $K$.   This interconnection is shown in
Figure~\ref{fig:lft}.

\begin{figure}[htb]
\centerline{\includegraphics[width=0.25\textwidth]{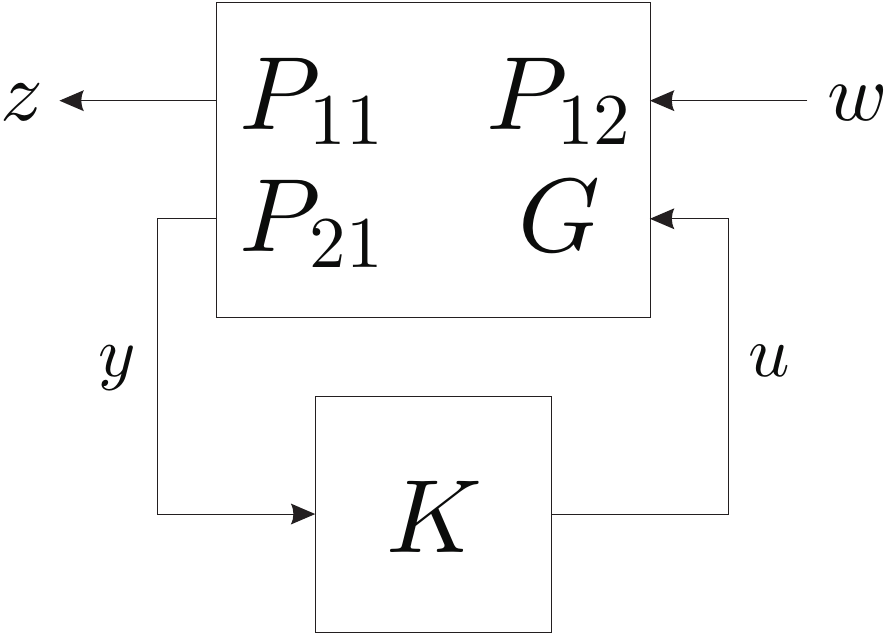}}
\caption{Linear fractional interconnection of $P$ and $K$}
\label{fig:lft}
\end{figure}

We suppose that there are $n_y$~sensor measurements and $n_u$~control actions,
and thus  partition the sensor measurements and
control actions as
\[
y=\bmat{y_1^T & \hdots & y_{n_y}^T}^T 
\qquad
u=\bmat{u_1^T & \hdots & u_{n_u}^T}^T 
\]
and then further partition $G$ and $K$ as
\[
G=
\small{
\bmat{G_{11} & \hdots & G_{1{n_u}} \\
  \vdots & & \vdots\\
G_{{n_y}1} & \hdots & G_{{n_y}{n_u}}}
}
\qquad
\small{
K=\bmat{K_{11} & \hdots & K_{1{n_y}} \\
  \vdots & & \vdots\\
K_{{n_u}1} & \hdots & K_{{n_u}{n_y}}}
}
\]


Given $A\in\R^{m\times n}$, we may write $A$ in term of its columns as
\[
A=\bmat{a_1 & \ldots & a_n}
\]
and then associate a vector
$\vecc(A)\in\R^{mn}$ defined by
\[
\vecc(A) \overset{def}{=}\bmat{a_1^T & \cdots & a_n^T}^T
\]
Further notation will be introduced as needed.

\subsection{Delays}
We define $\delop(\cdot)$ for a causal operator as the smallest amount of time in which an
input can affect its output.
For any causal linear time-invariant operator $\H:\genLe\rightarrow\genLe$, its delay is defined as:
\begin{align*}
\delop(\H) ~\overset{def}{=}~ \inf\{\tau\geq 0 ~\vert~ &z(\tau)\neq 0,
~z=\H(w), w\in\genLe,\\
&\text{ where }~w(t)=0~,~t\leq 0 \}
\end{align*}
and if $\H=0$, we consider its delay to be infinite. Here $\mathcal{L}_e$ is the domain
of $H$, which can be any extended $p$-normed Banach space of functions of non-negative continuous time with  co-domain in the reals. 

If  the map $\H$ has a well defined impulse response function $h$, then the delay of $H$ can be expressed as:
\[
\delop(\H) ~=~ \inf\{\tau\geq 0 ~\vert~ \h(\tau)\neq 0 \}\\
\]



\subsection{Sparsity}
We adopt the following notation to streamline our use of 
sparsity patterns and sparsity constraints.

\subsubsection{Binary algebra}
Let $\B=\{0,1\}$ represent the set of binary numbers. 
Given $x,y\in\B$, we define the following basic operations:
\begin{align*}
x+y ~&\overset{def}{=}~ 
\begin{cases}
  0, & \text{if } x=y=0 \\
  1, & \text{otherwise}
\end{cases}, \qquad x,y \in \mathbb{B}\\
xy ~&\overset{def}{=}~ 
\begin{cases}
  1, & \text{if } x=y=1 \\
  0, & \text{otherwise}
\end{cases}, \qquad x,y \in \mathbb{B}
\end{align*}

Given $X,Y\in\B^{\m\times n}$, we say that $X\leq Y$ holds if and only if $X_{ij}\leq
Y_{ij}$ for all $i,j$ satisfying $1 \leq i \leq \m$ and $1 \leq j \leq n$, where
$X_{ij}$ and $Y_{ij}$ are the entries at the $i$-th row and $j$-th column of the
binary matrices $X$ and $Y$, respectively.

Given $X,Y,Z\in\B^{\m\times n}$, these definitions lead to the following immediate consequences:
\begin{align}
Z=X+Y ~~&\Rightarrow~~ Z\geq X\label{b1}\\
X+Y=X ~~&\Leftrightarrow~~ Y\leq X\label{b2}\\
X\leq Y, ~Y\leq X ~~&\Leftrightarrow~~ X=Y\label{b3}
\end{align}

Given $X\in\B^{\m\times n}$, we use the following notation to represent the total number
of nonzero indeces in~$X$:
 \begin{equation*}
\NN(X) ~\overset{def}{=}~ \sum_{i=1}^\m \sum_{j=1}^n X_{ij}, \qquad X \in \B^{\m \times n}
\end{equation*}
with the sum taken in the usual way.

\subsubsection{Sparsity patterns}





Suppose that $\bin{A} \in \B^{\m \times n}$ is a binary matrix.
The following is the subspace of $\rp^{\m\times n}$ comprising the
transfer function matrices that satisfy the sparsity constraints imposed by $\bin{A}$:
\begin{align*}
\sparse(\bin{A}) \overset{def}{=}
\bl\{& B \in \rp^{\m\times n}  \ \vert \ B_{ij}(\jw) = 0 \text{ for all } i,j\\
&\text{ such that }   \bin{A}_{ij} =0
\text{ for almost all }\omega\in\R\br\}.
\end{align*}  
Conversely, given $B\in \rp^{\m\times n}$, we define $\pattern( B ) \overset{def}{=} \bin{A}$, where $\bin{A}$ is  the binary matrix given by:
\[
 \bin{A}_{ij} =
 \begin{cases} 0, & \text{if $B_{ij}(\jw) = 0$ for almost all $\omega\in\R$} \\
  1, & \text{otherwise}
\end{cases}, 
\]
for  $i\in\{\onetom\}, ~j\in\{\oneton\}.$

\section{Optimal Control Subject to Information Constraints}
\label{sec:probform}
In this section,  we give a detailed description of the two main types of information 
constraints adopted in this paper, namely, delay constraints and sparsity constraints. 


\subsection{Delay constraints}
\label{sec:delay_setup}

Consider  a plant comprising multiple subsystems which
may affect one another subject to propagation delays, and which may
communicate with one another with given transmission delays.
In what follows, we give a precise definition of these types of delays.

\subsubsection{Propagation Delays}

For any pair of subsystems $i$ and $j$, we define the propagation delay
$p_{ij}$ as the amount of time before a controller action at subsystem
$j$ can affect an output at subsystem $i$, as such:
%

%
\[
p_{ij}~\overset{def}{=}~\delop(G_{ij}),
\]
$\text{for all}~~i\in\{\onetony\}, ~j\in\{\onetonu\}$.

\subsubsection{Transmission Delays}

For any pair of subsystems $k$ and $l$, we define the (total)
transmission delay $t_{kl}$ as the minimum amount of time before
the controller of subsystem $k$ may use outputs from subsystem~$l$. 
Given these constraints, we can define the overall subspace of
admissible controllers
$S$ such that $K\in S$
if and only if the following holds:
\[
\delop(K_{kl}) ~\geq~ t_{kl}, 
\]
$\text{for all}~~k\in\{\onetonu\}, ~l\in\{\onetony\}$.

\subsection{Sparsity Constraints}
\label{sec:sparse_setup}

We now introduce the other main class of constraints we will consider in
this paper, where each control input may access certain sensor
measurements, but not others.

We represent sparsity constraints on the overall controller via a
binary matrix $\bin{K}\in\B^{n_u\times n_y}$. Its entries
can be interpreted as follows:
\[
\bin{K}_{kl} ~=~ 
\begin{cases}
1, & \text{if control input $k$}\\
& \text{may access sensor measurement $l$},\\
0, & \text{if not},
\end{cases} 
\]
for all $k\in\{\onetonu\}, ~l\in\{\onetony\}$.

The subspace of admissible controllers can be expressed as:
\[
S ~=~ \sparse(\bin{K}).
\]
From the quadratic invariance test introduced in \cite{rotkowitz_lall_tac05,qi_note}, we find that the relevant information about the plant is its sparsity pattern $\bin{G}$, obtained from:
\[
\bin{G} ~=~ \pattern(G)
\]
where $\bin{G}$ is interpreted as follows:
\[
\bin{G}_{ij} ~=~
\begin{cases}
1, & \text{if control input $j$}\\
& \text{affects sensor measurement $i$},\\
0, & \text{if not},
\end{cases}
\]
for all $i\in\{\onetony\}, ~j \in \{\onetonu\}$.

\subsection{Optimal Control Design Via Convex Programming}
\label{sec:setup}
Given a generalized plant $P$ and a subspace of appropriately dimensioned 
causal linear time-invariant controllers~$S$,
the following is a class of constrained optimal control problems:
\begin{equation}
\label{eqn:min_unstable_constrained_K}
\begin{aligned}
\underset{K}{\text{minimize}} \quad & \norm{f(P,K)} \\
\text{subject to} \quad & K \text{ stabilizes }P \\
& K \in S
\end{aligned}
\end{equation}
Here $\norm{\cdot}$ is any norm 
on the closed-loop map
chosen to encapsulate the control performance objectives.
The delays associated with dynamics propagating from one subsystem to
another, or the sparsity associated with them not propagating at all,
 are embedded in $P$.
The subspace of admissible controllers, $S$, has been defined to
encapsulate the constraints on how quickly information may be passed
from one subsystem to another (delay constraints) or whether it can be
passed at all (sparsity constraints).
We call the subspace~$S$ the \eemph{information constraint}.

Many decentralized control problems may be expressed in the form of
problem~\eqref{eqn:min_unstable_constrained_K}, including all of those
addressed in~\cite{siljak_1994, qi_voulgaris, rotkowitz_lall_tac05}.  In this paper, we focus
on the case where $S$ is defined by delay constraints or sparsity
constraints as discussed above.

This problem is made substantially more difficult in general by
the constraint that $K$ lie in the subspace~$S$.  Without this
constraint, the problem may be solved with many standard techniques.
Note that the cost
function $\norm{f(P,K)}$ is in general a non-convex function of
$K$.  No computationally tractable approach is known for solving
this problem for arbitrary~$P$ and~$S$.

\section{Quadratic Invariance}
\label{sec:qi}
In this section, we define quadratic invariance, and we give a brief
overview of related results, in particular, that if it holds then
convex synthesis of optimal decentralized controllers is possible.
\begin{defn}
 Let a causal linear time-invariant plant, represented via a transfer function matrix $G$
 in $\rp^{n_y \times n_u}$, be given. If $S$ is a subset of $\rp^{n_u \times n_y}$ 
then $S$ is called  \eemph{quadratically invariant}
 under~$G$ if the following inclusion holds:
  \[
   KGK\in S \qquad \text{for all } K\in S.
  \]
\end{defn} 


It was shown in~\cite{rotkowitz_lall_tac05} that if $S$ is a closed
subspace and
$S$ is quadratically invariant under~$G$, then
with a change of variables,
problem~\eqref{eqn:min_unstable_constrained_K} is equivalent to the
following optimization problem
\begin{equation}
\label{eqn:min_convex_Q}
\begin{aligned}
\underset{Q}{\text{minimize}} \quad & \norm{T_1 - T_2 Q T_3} \\
\text{subject to} \quad & Q\in \mathcal{R}\bspace{H}_\infty\\
& Q \in S \\
\end{aligned}
\end{equation}
where $T_1,T_2,T_3\in\rhi$. Here $\rhi$ is used to indicate that $T_1$, $T_2$, $T_3$ and $Q$ are 
proper transfer function matrices with no poles in $\mathbb{C}_+$ (stable) . 


The optimization problem in (\ref{eqn:min_convex_Q}) is convex. We may solve it to find the
optimal $Q$, and then recover the optimal $K$ for our original problem as stated in (\ref{eqn:min_unstable_constrained_K}). If the norm of interest is the $\Htwo$-norm, it was shown
in~\cite{rotkowitz_lall_tac05}
that the problem
can be further reduced to an unconstrained optimal control problem and
then solved with standard software. Similar
results have been
achieved~\cite{qi_note}
for function spaces beyond $\mathcal{L}_e$ as well, also showing that quadratic invariance
allows optimal linear decentralized control problems to be recast as
convex optimization problems.

The main focus of this paper is thus characterizing information constraints $S$ which
are as close as possible to a pre-selected one, and for which
$S$ is quadratically invariant under the plant $G$.

\subsection{QI - Delay Constraints}
\label{sec:qi_delays}

For the case of delay constraints, it was shown
in~\cite{rotkowitz_cdc05} that 
 a necessary and sufficient condition for quadratic invariance
is
%


\begin{equation}
\label{eqn:tptt}
t_{ki}+p_{ij}+t_{jl}~\geq~ t_{kl}, 
\end{equation}
for all $i,l \in \{ \onetony \},$ and  all $j,k \in \{ \onetonu \}$.


Note that it was further shown in~\cite{rotkowitz_cdc05} that if
we consider the typical case of $n$ subsystems, each with its own
controller, such that $n=n_y=n_u$, and if
 the transmission delays satisfy the triangle inequality, then the
quadratic invariance test can be further reduced to the following
inequality:
\begin{equation}
\label{eqn:pt}
p_{ij}~\geq~ t_{ij}, 
\end{equation}
for all $i,j \in \{ \oneton \}$; 
that is, the communication between any two nodes needs to be as fast
as the propagation between the same pair of nodes.

\subsection{QI - Sparsity Constraints}
\label{sec:qi_sparsity}

%


For the case of sparsity constraints, it was shown
in~\cite{rotkowitz_lall_tac05} that a necessary and sufficient
condition for quadratic invariance is
\begin{equation}
\label{eqn:kgkkbin}
K^\text{bin}_{ki}~G^\text{bin}_{ij}~K^\text{bin}_{jl}~(1-K^\text{bin}_{kl})
=  0,
\end{equation}
for all $i,l \in \{ \onetony \}$, and all  $j,k \in \{ \onetonu \}$.

It can be shown that this is equivalent to Condition~\eqref{eqn:tptt}
if we let
\begin{align}
t_{kl} ~&=~ \begin{cases}
\bigdelay, & \text{if }\bin{K}_{kl}=0\\
0, & \text{if }\bin{K}_{kl}=1
\end{cases}\label{eqn:t_to_bin}, 
\intertext{for all  $k \in \{ \onetonu \}$, and all $l \in \{ \onetony
  \}$, and let}
p_{ij} ~&=~ \begin{cases}
\bigdelay, & \text{if }\bin{G}_{ij}=0\\
0, & \text{if }\bin{G}_{ij}=1
\end{cases}\label{eqn:p_to_bin}, 
\end{align}
for all $i \in  \{ \onetony \}$, and all $j \in \{ \onetonu \}$,
for any $\bigdelay > 0$, the interpretation being that a sparsity
constraint can be thought of as a large delay, and a lack thereof can
be thought of as no delay.

\section{Closest QI Constraint}
\label{sec:closest}
We now address the main question of this paper, which is finding the
closest constraints when the above conditions fail; that is, when the
original problem is not quadratically invariant.

\subsection{Closest - Delays}
\label{sec:closest_delays}
Suppose that we are given propagation delays $\{ \porig_{ij} \}_{i=1,j=1}^{n_y,n_u}$ and transmission
delays $\{ \torig_{kl} \}_{k=1,l=1}^{n_u,n_y}$, and that they do not satisfy
Condition~\eqref{eqn:tptt}.  The problem of finding the closest
constraint set, that is, the transmission delays $\{ \tvar_{kl} \}_{k=1,l=1}^{n_u,n_y}$
that are closest to  $\{ \torig_{kl} \}_{k=1,l=1}^{n_u,n_y}$  while satisfying \eqref{eqn:tptt}, can be set up as follows:
\begin{equation}
\label{eqn:closest_trans}
\begin{aligned}
\underset{t}{\text{minimize}} \quad & \norm{\vecc(\tvar-\torig)} & \\
\text{subject to} \quad & \tvar_{ki}+\porig_{ij}+t_{jl}~\geq~\tvar_{kl}, 
\qquad  &\forall~i,j,k,l,\\
    & t_{kl} \geq 0,  \qquad &\forall~k,l.
\end{aligned}
\end{equation}


This is a convex optimization problem in the new transmission
delays~$\tvar$.  The norm is arbitrary, and may be chosen to
encapsulate whatever notion of closeness is most appropriate.  If the
1-norm is chosen, corresponding to minimizing the sum of the
differences in transmission delays, or the $\infty$-norm is chosen,
corresponding to minimizing the largest difference, then the problem
may be cast as a linear program~(LP).

If we want to find the closest quadratically invariant set, which is a
superset of the original set, so that we may obtain a lower bound to
the solution of the main
problem~\eqref{eqn:min_unstable_constrained_K}, then we simply add the
constraint $\tvar_{kl}\leq\torig_{kl}$ for all~$k,l$, and the problem remains convex (or
remains an LP). Note that if we follow the aforementioned procedure and choose the 1-norm,
then the objective is equivalent to maximizing the total delay sum ~$\sum_{k=1}^{n_u}\sum_{l=1}^{n_y}\tvar_{kl}$.

Similarly, if we want to find the closest quadratically invariant set which is a
subset of the original set, so that we may obtain an upper bound to
the solution of the main
problem~\eqref{eqn:min_unstable_constrained_K}, then we simply add the
constraint $\tvar_{kl}\geq\torig_{kl}$ for all~$k,l$, and the problem remains convex (or
remains an LP). For this procedure and if we choose the 1-norm,
then the objective is equivalent to minimizing the total delay sum ~$\sum_{k=1}^{n_u}\sum_{l=1}^{n_y}\tvar_{kl}$.

\subsection{Closest - Sparsity}
Now suppose that we want to construct the closest
quadratically invariant set, superset, or subset, defined by sparsity
constraints.  We can recast a pre-selected sparsity constraint on
the controller $\bin{K}$ and a given sparsity pattern of the plant $\bin{G}$
as in~\eqref{eqn:t_to_bin},\eqref{eqn:p_to_bin}, and then set up
problem~\eqref{eqn:closest_trans}. 
 The only problem is that for the
resulting solution to correspond to a sparsity constraint, we need to
add the binary constraints $\tvar_{kl}\in\{0,\bigdelay\}$ for all $k,l$,
and this destroys the convexity of the problem.

\subsubsection{Sparsity Superset}
Consider first finding the closest quadratically invariant superset of
the original constraint set; that is, the sparsest quadratically
invariant set for which all of the original connections
$y_l\rightarrow u_k$ are still in place.

This is equivalent to solving the above
problem~\eqref{eqn:closest_trans} with $\tvar_{kl}\leq\torig_{kl}$ for
all~$k,l$, and with the binary constraints, an intractable
combinatorial problem, but we present an algorithm which solves it and
terminates in a fixed number of steps.

We can write the problem as
\begin{equation}
\label{eqn:closest_super_sparse}
\begin{aligned}
\underset{Z \in \B ^{n_u \times n_y} } {\text{minimize}} \quad & \NN(Z) \\
\text{subject to} \quad & Z\bin{G}Z ~\leq~ Z\\
& \bin{K} ~\leq~ Z
\end{aligned}
\end{equation}
where additions and multiplications are as defined for the binary
algebra in the preliminaries, and where we will wish to use the
information constraint $S=\sparse(Z)$.
The objective is defined to give us the
sparsest possible solution, the first constraint ensures that the
constraint set associated with the solution is quadratically invariant
with respect to the plant, and the last constraint requires the
resulting set of controllers to be able to access any information that
could be accessed with the original constraints. Let the optimal
solution to this optimization problem be denoted as
$Z^*\in\B^{n_u\times n_y}$.

Define a sequence of sparsity constraints $\{Z_m\in\B^{n_u\times n_y},
~m\in\N\}$ given by
\begin{align}
\label{eqn:sequence_init}
Z_0 ~&=~ \bin{K}\\
\label{eqn:sequence_step}
Z_{m+1} ~&=~ Z_m+Z_m\bin{G}Z_m,\quad ~m\geq 0
\end{align}
again using the binary algebra.

Our main result will be that this sequence converges to $Z^*$, and
that it does so in $\log_2 n$ iterations.  
We first prove several
preliminary lemmas, and start with a lemma elucidating which terms comprise which elements of
the sequence.
\begin{lem}
\label{lem:sequence_components}
\begin{equation}
\label{eqn:sequence_components}
Z_m ~=~ \sum_{s=0}^{2^m-1}\kgkbin^s \qquad\forall~m\in\N
\end{equation}
\end{lem}
\begin{proof}
For $m=0$, this follows immediately from~\eqref{eqn:sequence_init}.
We then assume that~\eqref{eqn:sequence_components} holds for a given
$m\in\N$, and consider $m+1$.  Then,
\begin{multline*}
Z_{m+1} ~=~ \sum_{i=0}^{2^m-1}\kgkbin^i +\\
\left(\sum_{k=0}^{2^m-1}\kgkbin^k\right) 
\bin{G}
\left(\sum_{l=0}^{2^m-1}\kgkbin^l\right).
\end{multline*}
All terms on the R.H.S. are of the form $\kgkbin^s$
for various $s\in\N$.  Choosing  $0\leq i\leq 2^m-1$ gives 
$0\leq s\leq 2^m-1$, and choosing $k=2^m-1$ with $0\leq l\leq 2^m-1$
gives $2^m\leq s\leq (2^m-1)+1+(2^m-1)=2^{m+1}-1$. This last term is
the highest order term, so we then have
$Z_{m+1}=\sum_{s=0}^{2^{m+1}-1}\kgkbin^s$ and the
proof follows by induction.
\end{proof}

We now give a lemma showing how many of these terms need to be considered.
\begin{lem}
\label{lem:num_terms}
The following holds for $n=\min\{n_y,n_u\}$:
\begin{equation}
\label{eqn:num_terms}
\kgkbin^r ~\leq~
\sum_{s=0}^{n-1}\kgkbin^s
\quad~~\forall~~ r\in\N.
\end{equation}
\end{lem}
\begin{proof}
Follows immediately from~\eqref{b1} for $r\leq n-1$. Now consider
$r\geq n$,  $k\in\{\onetonu\}$, $l\in\{\onetony\}$.
Then
$
[\kgkbin^r]_{kl}=\sum
\bin{K}_{ki_1}\bin{G}_{i_1j_1}\bin{K}_{j_1i_2}\bin{G}_{i_2j_2}
\cdots\bin{G}_{i_rj_r}\bin{K}_{j_rl}
$
where the sum is taken over all possible $i_\alpha\in\{\onetony\}$ and
$j_\alpha\in\{\onetonu\}$.
Consider an arbitrary such summand term that is equal to 1, and note
that each component term must be equal to 1.  

If $n=n_y$ (i), then
by the pigeonhole principle either $\exists\alpha$ s.t. $i_\alpha=l$ (i.a), or
$\exists\alpha,\beta$, with $\alpha\neq\beta$,
s.t. $i_\alpha=i_\beta$ (i.b).  In case (i.a), we have
 $\bin{K}_{ki_1}\bin{G}_{i_1j_1}\cdots\bin{G}_{i_{\alpha-1}j_{\alpha-1}}\bin{K}_{j_{\alpha-1}l}=1$,
or in case (i.b), we have 
$\bin{K}_{ki_1}\cdots\bin{K}_{j_{\alpha-1}i_\alpha}\bin{G}_{i_{\beta}j_\beta}\cdots\bin{K}_{j_rl}=1$.
In words, we can bypass the part of the path that merely took
$y_{i_\alpha}$ to itself, leaving a shorter path that still connects
$y_l\rightarrow u_k$.

Similarly, if $n=n_u$ (ii), then either $\exists\alpha$
s.t. $j_\alpha=k$ (ii.a), or $\exists\alpha,\beta$ with
$\alpha\neq\beta$ s.t. $j_\alpha=j_\beta$ (ii.b). In case (ii.a), we
have
$\bin{K}_{ki_{\alpha+1}}\bin{G}_{i_{\alpha+1}j_{\alpha+1}}\cdots\bin{K}_{j_rl}=1$,
or in case (ii.b), we have
$\bin{K}_{ki_1}\cdots\bin{G}_{i_{\alpha}j_\alpha}\bin{K}_{j_{\beta}i_{\beta+1}}\cdots\bin{K}_{j_rl}=1$,
where we have now bypassed the part of the path taking $u_{j_\alpha}$
to itself to leave a shorter path.

We have shown that, $\forall~r\geq n$, any non-zero component term of
$\kgkbin^r$ has a corresponding non-zero term of strictly lower order,
and the result follows.
\end{proof}

We now prove another
preliminary lemma showing that the optimal solution can be no more
sparse than any element of the sequence.
\begin{lem}
\label{lem:sparser_than_opt}
For $Z^*\in\B^{n_u \times n_y}$ and the sequence $\{Z_m\in\B^{n_u \times n_y},
~m\in\N\}$ defined as above, the following holds:
\begin{equation}
\label{eqn:sparser_than_opt}
Z^* ~\geq~ Z_m,  \quad ~m\in\N
\end{equation}
\end{lem}
\begin{proof}
First, $Z^* \geq Z_0 = \bin{K}$ is given by the satisfaction of the
last constraint of~\eqref{eqn:closest_super_sparse}, and it just
remains to show the inductive step. 

Suppose that $Z^*\geq Z_m$ for some $m\in\N$. It then follows that 
\[Z^*+Z^*\bin{G}Z^* ~\geq~ Z_m+Z_m\bin{G}Z_m.\]
From the first constraint of~\eqref{eqn:closest_super_sparse}
and~\eqref{b2} we know that the left hand-side is just $Z^*$, and then
using the definition of our sequence, we get
$Z^* ~\geq~ Z_{m+1}$
and this completes the proof.
\end{proof}

We now give a subsequent lemma, showing that if the sequence does
converge, then it has converged to the optimal solution.
\begin{lem}
\label{lem:conv_is_opt}
If $Z_{m^*}=Z_{m^*+1}$ for some $m^*\in\N$, then $Z_{m^*}=Z^*$.
\end{lem}
\begin{proof}
If $Z_{m^*}=Z_{m^*+1}$, then $Z_{m^*}=Z_{m^*} +Z_{m^*}\bin{G}Z_{m^*}$, and it
follows from~\eqref{b2} that $Z_{m^*}\bin{G}Z_{m^*}\leq Z_{m^*}$. 
Since $Z_{m+1}\geq Z_m$ for all $m\in\N$, it also follows that
$Z_{m}\geq Z_0=\bin{K}$ for all $m\in\N$.  Thus the two constrains
of~\eqref{eqn:closest_super_sparse} are satisfied for $Z_{m^*}$.

Since $Z^*$ is the sparsest binary matrix satisfying these
constraints, it follows that $Z^*\leq Z_{m^*}$.
Together with Lemma~\ref{lem:sparser_than_opt} and equation~\eqref{b3}, it follows that $Z_{m^*}=Z^*$.
\end{proof}

We now give the main result: that the sequence converges, that it does
so in $\log_2 n$ steps, and that it achieves the optimal solution to
our problem.

\begin{thm} 
\label{thm:closest_super_sparse_conv}
The problem specified in (\ref{eqn:closest_super_sparse}) has an unique optimal solution $Z^*$
satisfying:
\begin{equation} 
Z_{m^*} ~=~ Z^*
\end{equation} where $m^*=\lceil\log_2 n\rceil$ and where $n=\min\{n_u,n_y\}$.
\end{thm}
\begin{proof}
$Z_{m^*} = \sum_{s=0}^{2^{m^*}-1}\kgkbin^s$
from Lemma~\ref{lem:sequence_components}, and then 
$Z_{m^*} =  \sum_{s=0}^{n-1}\kgkbin^s$ from Lemma~\ref{lem:num_terms}
since $2^{m^*}\geq n$.  Similarly, 
$Z_{m^*+1} = \sum_{s=0}^{2^{m^*+1}-1}\kgkbin^s
=  \sum_{s=0}^{n-1}\kgkbin^s$, and thus $Z_{m^*}=Z_{m^*+1}$ and the
result follows from Lemma~\ref{lem:conv_is_opt}.
\end{proof}

\paragraph{Delay Superset (Revisited)}
Suppose again that we wish to find the closest superset defined by
delay constraints. This can be found by convex optimization as
described in Section~\ref{sec:closest_delays}.  It can also be found
with the above algorithm, where $\bin{G}$ is replaced with a matrix of
the propagation delays, $\bin{K}$ with the given transmission delays,
and where the binary algebra is replaced with the
$(\min,+)$ algebra.
The $Z_m$ matrices then hold the transmission delays at each
iteration, and with the appropriate inequalities flipped (since
finding a superset means decreasing transmission delays), the proofs
of the lemmas and the convergence theorem follow almost identically.

\subsubsection{Sparsity Subset}
We now notice an interesting asymmetry.
For the case of delay constraints, if we were interested in finding
the most restrictive superset (for a lower bound), or the least
restrictive subset (for an upper bound), we simply flipped the sign of
our final constraint, and the problem was convex either way.  When we
instead consider sparsity constraints, the binary constraint ruins the
convexity, but we see that in the former (superset) case we can still
find the closest constraint in a fixed number of iterations in
polynomial time; however, for the latter (subset) case, there is no
clear way to ``flip'' the algorithm.

This can be understood as follows.  If there exist indeces $i,j,k,l$
such that $\bin{K}_{ki}=\bin{G}_{ij}=\bin{K}_{jl}=1$, but
  $\bin{K}_{kl}=0$; that is, indeces for which
  condition~\eqref{eqn:kgkkbin} fails, then the above algorithm resets
  $\bin{K}_{kl}=1$.
In other words, if there is an indirect connection from
$y_l\rightarrow u_k$, but not a direct connection, it hooks up the
direct connection.

But now consider what happens if we try to develop an algorithm that
goes in the other direction, that finds the least sparse constraint
set which is more sparse than the original.  If we again have indeces
for which condition~\eqref{eqn:kgkkbin} fails, then we need to disconnect
the indirect connection, but it's not clear if we should set
$\bin{K}_{ki}$ or $\bin{K}_{jl}$ to zero, since we could do either.
The goal is, in principle, to disconnect the link that will ultimately
lead to having to make the fewest subsequent disconnections, so that
we end up with the closest possible constraint set to the original.

We suggest some methods for dealing with this problem.  It is likely
that they can be greatly improved upon, but are meant as a first cut
at a reasonable polynomial time algorithm to find a close sparse
subset.

For the first heuristic, we set up transmission delays and propagation
delays as in~\eqref{eqn:t_to_bin} and~\eqref{eqn:p_to_bin}, and then
instead of adding the binary constraint and making the problem
non-convex, add the relaxed constraint $0\leq\tvar_{kl}\leq\bigdelay$ for all
$k,l$, and solve the resulting convex problem.
Then, for a set of indeces violating condition~\eqref{eqn:kgkkbin},
set $\bin{K}_{ki}$ to zero if $\tvar^*_{ki}\geq\tvar^*_{jl}$, and set
$\bin{K}_{jl}$ to zero otherwise, before re-solving the convex
problem. 
The motivation is that we
disconnect the one that has a larger delay, that is, which is more
constrained, in the case where we allowed varying degrees of constraint.

The relaxed problem could instead be solved with increasing penalties on
the entropy of $\vecc(t/\bigdelay)$, to approach a binary solution,
as in the study of probability collectives~\cite{wolpert_pc_acs}.
 This method
has the benefit that it could be used to find a close sparse set
or subset.

For the second heuristic, we more directly keep track of how many
indirect connections are associated with a direct connection.  Define
this weight as 
$w_{kl}=\sum_{i=1}^{n_y}\sum_{j=1}^{n_u}\bin{K}_{ki}\bin{G}_{ij}\bin{K}_{jl}$
thus giving the amount of 3-hop connections from $y_l\rightarrow
u_k$.  This is a crude measure of how many subsequent disconnections
we will have to make to obtain quadratic invariance if we were to
disconnect a direct path from $y_l\rightarrow u_k$.  Then, given
indeces for which condition~\eqref{eqn:kgkkbin} is violated, we set
$\bin{K}_{ki}$ to zero if $w_{ki}\leq w_{jl}$, and set $\bin{K}_{jl}$
to zero otherwise.

Note that for either heuristic, we have many options for how often to reset the guiding variables,
that is, to re-solve the convex program or recalculate the weights,
such as after each disconnection, or after each pass through all
$n_un_y$ indeces.

It has been noticed that some of the quadratically invariant constraints for
certain classes of problems, including sparsity, may be thought of as partially ordered
sets~\cite{shah2008poset}. This raises the possibililty that work in
that area, such as~\cite{cardinal2009poset}, may be leveraged to more
efficiently find the closest sparse sets or subsets.

\section{Nonlinear Time-Varying Control}
\label{sec:nltv}


It was shown in~\cite{rotkowitz_acc06} that if we consider the design
of possibly nonlinear, possibly time-varying (but still causal)
controllers to stabilize
possibly nonlinear, possibly time-varying (but still causal) plants,
then while the quadratic invariance results no longer hold,
the following condition
  \[
   K_1(I\pm GK_2)\in S \qquad \text{for all } K_1,K_2\in S
  \]
similarly allows for a convex parameterization of all stabilizing
controllers subject to the given constraint.

This condition is equivalent to quadratic invariance when $S$ is
defined by delay constraints or by sparsity constraints, and so the
algorithms in this paper may also be used to find the closest
constraint for which this is achieved.

\section{Numerical Examples}
\label{sec:numex}
We present some numerical examples of the algorithms developed in this
paper.

\subsection{Example - Sparsity Constraints}
\label{sec:numex_sparsity}
We start this section by finding the closest quadratically invariant superset, with respect to
the following sparsity patterns of two plants with
four subsystems each ($n=4$):
\begin{equation}
\label{eqn:numex_sparse_plants}
\bin{G}_{I} = \footnotesize{\left[ \begin{array}{cccc} 1 & 0 & 0 & 0 \\ 1 & 1 & 0 & 0 \\ 0 & 1 & 1 & 1 \\ 0 & 0 & 0 & 1 \end{array}
\right]}
\hskip 10pt
\bin{G}_{II} = \footnotesize{\left[ \begin{array}{cccc} 1 & 0 & 0 & 0 \\ 1 & 1 & 0 & 0 \\ 0 & 1 & 1 & 0 \\ 0 & 0 & 1 & 1 \end{array}
\right]}
\end{equation}

The first sparsity pattern in~\eqref{eqn:numex_sparse_plants}
represents a plant where the first two control inputs effect not only
their own subsystems, but also the subsequent subsystems, and where
the last control input effects not only its own subsystem, but also
the preceding subsystem.  The second sparsity pattern represents a
plant where each control input effects its own subsystem and the
subsequent subsystem, which also corresponds to the
open daisy-chain configuration.
 Now consider an initial proscribed controller configuration where the
 controller for each subsystem has access only to the measurement from
 its own subsystem:
\begin{equation}
\bin{K} = \footnotesize{\left[ \begin{array}{cccc} 1 & 0 & 0 & 0 \\ 0 & 1 & 0 & 0 \\ 0 & 0 & 1 & 0 \\ 0 & 0 & 0 & 1 \end{array}
\right]}
\end{equation}
that is, where the controller is block diagonal.
Using the algorithm specified in
(\ref{eqn:sequence_init})-(\ref{eqn:sequence_step}) 
we arrive at:
\begin{equation}
\label{eqn:numex_sparse_supersets}
Z^*_{I} = \footnotesize{\left[ \begin{array}{cccc} 1 & 0 & 0 & 0 \\ 1 & 1 & 0 & 0 \\ 1 & 1 & 1 & 1 \\ 0 & 0 & 0 & 1 \end{array}
\right]}
\hskip 30pt
Z^*_{II} = \footnotesize{\left[ \begin{array}{cccc} 1 & 0 & 0 & 0 \\ 1 & 1 & 0 & 0 \\ 1 & 1 & 1 & 0 \\ 1 & 1 & 1 & 1 \end{array}
\right]}
\end{equation}
where $Z^*_{I}$ and $Z^*_{II}$ denote the optimal solution of (\ref{eqn:closest_super_sparse})
as applied to $\bin{G}_{I}$ and $\bin{G}_{II}$, respectively, and thus
represent the sparsity constraints of the closest quadratically
invariant supersets of the set of block diagonal controllers. 
We see that a quadratically invariant set of controllers for the first
plant (which contains all block diagonal controllers) has to have the
same sparsity pattern as the plant, and an additional link from the
first measurement to the third controller.
We then see that any quadratically invariant set for the open
daisy-chain configuration which contains the diagonal will have to be
lower triangular.

\subsection{Example - Delay Constraints}
\label{sec:numex_delays}
We consider $n=4$ subsystems, with the following given propagation
delays and the following proscribed transmission delays, all chosen as
random uniform integers from 0 to 9:
\[
\porig~=~
\footnotesize{\left[\begin{array}{cccc} 9 & 0 & 8 & 4\\ 0 & 7 & 8 & 7\\ 3 & 5 & 7 & 1\\ 5 & 5 & 3 & 1 \end{array}\right]}
\hskip 30pt
\torig~=~
\left[\begin{array}{cccc} 2 & 3 & 6 & 5\\ 5 & 2 & 2 & 9\\ 9 & 8 & 0 &  0\\ 7 & 9 & 8 & 5 \end{array}\right]
\]
We then display in Table~\ref{delay_table} the difference between the delays of the closest
quadratically invariant subset, set, and superset from the given
transmission delays $(t-\torig)$, as measured in the vector \mbox{1-norm}, \mbox{2-norm},
and \mbox{$\infty$-norm}.  
These were computed by solving the convex problem of
Section~\ref{sec:closest_delays}, which was done in Matlab using the
\texttt{CVX} package~\cite{cvx_toolbox}, and verified using \texttt{linprog()}
for the 1-norm and $\infty$-norm.
\begin{table*}[htb]
\begin{center}
\begin{tabular}{c||c|c|c}
&Closest Subset & Closest Set & Closest Superset
\\ \hline\hline
$\norm{\cdot}_1$
&
\footnotesize{$\begin{array}{cccc} 1.74 & 0 & 0 & 0.57\\ 0 & 1.43 & 0.68 & 0\\ 0 & 0 & 2.26 & 0.73\\ 0.27 & 0 & 0 & 0.32 \end{array} $}
&
\footnotesize{$\begin{array}{cccc} 1.77 & 0.08 & -0.06 & 0.13\\ 0 & 0.87 & 0.17 & -1.0\\ -0.65 & -0.1 & 1.81 & 0.29\\ 0.06 & 0 & 0 & 0 \end{array} $}
&
\footnotesize{$\begin{array}{cccc} 0 & 0 & -2.0 & 0\\ -1.0 & 0 & 0 & -2.0\\ -4.0 & -2.0 & 0 & 0\\ 0 & 0 & 0 & 0 \end{array} $}
\\ \hline
$\norm{\cdot}_2$
&
\footnotesize{$\begin{array}{cccc} 2.0 & 0 & 0 & 1.0\\ 0 & 1.0 & 0.5 & 0\\ 0 & 0 & 2.0 & 0.5\\ 0.5 & 0 & 0 & 0.5 \end{array} $}
&
\footnotesize{$\begin{array}{cccc} 1.43 & 0.27 & -0.27 & 0.65\\ 0 & 0.68 & 0.3 & -0.67\\ -1.15 & -0.3 & 1.42 & 0.02\\ 0 & 0 & 0 & 0.03 \end{array} $}
&
\footnotesize{$\begin{array}{cccc} 0 & 0 & -2.0 & 0\\ -1.0 & 0 & 0 & -2.0\\ -4.0 & -2.0 & 0 & 0\\ 0 & 0 & 0 & 0 \end{array} $}
\\ \hline
$\norm{\cdot}_\infty$
&
\footnotesize{$\begin{array}{cccc} 2.0 & 1.25 & 0.56 & 1.29\\ 0.81 & 1.57 & 1.45 & 0.31\\ 0 & 0.36 & 2.0 & 1.48\\ 1.15 & 0.66 & 0.66 & 1.24 \end{array} $}
&
\footnotesize{$\begin{array}{cccc} 1.33 & 0.64 & -0.55 & 0.6\\ -0.05 & 1.07 & 0.9 & -0.89\\ -1.33 & -0.78 & 1.33 & 0.99\\ 0.44 & -0.27 & -0.29 & 0.66 \end{array} $}
&
\footnotesize{$\begin{array}{cccc} 0 & -0.65 & -2.92 & -0.69\\ -2.12 & -0.21 & -0.32 & -3.48\\ -4.0 & -3.44 & 0 & 0\\ -1.07 & -2.6 & -2.66 & -0.59 \end{array}$}
\end{tabular}
\caption{Distance to Closest Quadratically Invariant Delay Constraints}
\label{delay_table}
\end{center}
\end{table*}

The delays have to be increased to reach the closest QI subsets, so
the first column contains only nonnegative numbers, and the delays are
decreased to get to the closest supersets, so the last column contains
only nonpositive numbers, and the delays may be moved in either
direction to get to the closest QI set in the middle column.  Finding
the closest QI superset is actually the same in any norm, as each
delay between a given measurement and control action is set to the
fastest indirect delay between those two signals. We indeed see that
the superset is the same for the 1-norm and 2-norm.  This same set of
delays would also solve the problem for the $\infty$-norm, but it has
selected a matrix with some smaller delays, with the same maximum
change of 4.  This shows the problem with the lack of uniquness that
often arises when optimizing the $\infty$-norm: it is indifferent to
further moving the delays that have not been moved the maximum
amount. Thus it should never be used to find the closest superset, and
when it is appropriate to control it in finding a closest set or
subset, it should be used in conjunction with another norm as well.
We see that this example produces the same level of sparsity in the
delay differences to closest subset and set for the 1-norm and 2-norm,
though minimizing the 1-norm generally produces sparse
solutions~\cite{donoho2006most}, and should be the norm to choose here when one
wishes to alter as few of the delays as necessary.

\section*{Acknowledgment}
The authors would like to thank Randy Cogill for useful discussions
related to the delay constraints. 


\section{Conclusions}
\label{sec:conc}

The overarching goal of this paper is the design of
linear time-invariant,
decentralized controllers for plants comprising
dynamically coupled subsystems.
Given pre-selected constraints on the controller which capture the
decentralization being imposed, we addressed the question of finding
the closest constraint which is quadratically invariant under the
plant.
Problems subject to such constraints are amenable to convex synthesis,
so this is important for bounding the optimal solution to the original
problem.

We focused on two particular classes of this problem. The first is where the
  decentralization imposed on the controller is specified by delay
  constraints; that is, information is passed between subsystems with
  some given delays, represented by a matrix of transmission delays.
The second is where the decentralization imposed on the controller
  is specified by sparsity constraints; that is, each controller can
  access information from some subsystems but not others, and this is
  represented by a binary matrix.

For the delay constraints, we showed that finding the closest
quadratically invariant constraint can be set up as a convex
optimization problem.  We further showed that finding the closest
superset; that is, the closest set that is less restrictive than the
pre-selected one, to get lower bounds on the original problem, is also
a convex problem, as is finding the closest subset.

For the sparsity constraints, the convexity is lost, but we provided
an algorithm which is guaranteed to give the closest quadratically
invariant superset in at most $\log_2 n$ iterations, where $n$ is the
number of subsystems.
 We also discussed methods to give close
quadratically invariant subsets.

\bibliographystyle{IEEEtran}

\bibliography{IEEEabrv,mikenunotn}

\begin{thebibliography}{10}
\providecommand{\url}[1]{#1}
\csname url@samestyle\endcsname
\providecommand{\newblock}{\relax}
\providecommand{\bibinfo}[2]{#2}
\providecommand{\BIBentrySTDinterwordspacing}{\spaceskip=0pt\relax}
\providecommand{\BIBentryALTinterwordstretchfactor}{4}
\providecommand{\BIBentryALTinterwordspacing}{\spaceskip=\fontdimen2\font plus
\BIBentryALTinterwordstretchfactor\fontdimen3\font minus
  \fontdimen4\font\relax}
\providecommand{\BIBforeignlanguage}[2]{{%
\expandafter\ifx\csname l@#1\endcsname\relax
\typeout{** WARNING: IEEEtran.bst: No hyphenation pattern has been}%
\typeout{** loaded for the language `#1'. Using the pattern for}%
\typeout{** the default language instead.}%
\else
\language=\csname l@#1\endcsname
\fi
#2}}
\providecommand{\BIBdecl}{\relax}
\BIBdecl

\bibitem{witsenhausen_1971}
H.~S. Witsenhausen, ``Separation of estimation and control for discrete time
  systems,'' \emph{Proceedings of the IEEE}, vol.~59, no.~11, pp. 1557--1566,
  1971.

\bibitem{sandell_1978}
N.~Sandell, P.~Varaiya, M.~Athans, and M.~Safonov, ``Survey of decentralized
  control methods for large scale systems,'' \emph{IEEE Transactions on
  Automatic Control}, vol.~23, no.~2, pp. 108--128, February 1978.

\bibitem{witsenhausen_1968}
H.~S. Witsenhausen, ``A counterexample in stochastic optimum control,''
  \emph{SIAM Journal of Control}, vol.~6, no.~1, pp. 131--147, 1968.

\bibitem{rotkowitz_lall_tac05}
M.~Rotkowitz and S.~Lall, ``A characterization of convex problems in
  decentralized control,'' \emph{IEEE Transactions on Automatic Control},
  vol.~51, no.~2, pp. 274--286, February 2006.

\bibitem{qi_note}
------, ``Affine controller parameterization for decentralized control over
  {Banach} spaces,'' \emph{IEEE Transactions on Automatic Control}, vol.~51,
  no.~9, pp. 1497--1500, September 2006.

\bibitem{voulgaris_2001}
P.~G. Voulgaris, ``A convex characterization of classes of problems in control
  with specific interaction and communication structures,'' in \emph{Proc.
  American Control Conference}, 2001, pp. 3128--3133.

\bibitem{siljak_1994}
D.~D. Siljak, \emph{Decentralized control of complex systems}.\hskip 1em plus
  0.5em minus 0.4em\relax Academic Press, Boston, 1994.

\bibitem{qi_voulgaris}
X.~Qi, M.~Salapaka, P.~Voulgaris, and M.~Khammash, ``Structured optimal and
  robust control with multiple criteria: A convex solution,'' \emph{IEEE
  Transactions on Automatic Control}, vol.~49, no.~10, pp. 1623--1640, 2004.

\bibitem{rotkowitz_cdc05}
M.~Rotkowitz, R.~Cogill, and S.~Lall, ``A simple condition for the convexity of
  optimal control over networks with delays,'' in \emph{Proc. {IEEE} Conference
  on Decision and Control}, 2005, pp. 6686--6691.

\bibitem{wolpert_pc_acs}
D.~Wolpert, C.~Strauss, and D.~Rajnarayan, ``Advances in distributed
  optimization using probability collectives,'' \emph{Advances in Complex
  Systems}, vol.~9, no.~4, pp. 383--436, 2006.

\bibitem{shah2008poset}
P.~Shah and P.~Parrilo, ``{A partial order approach to decentralized
  control},'' in \emph{Proc. {IEEE} Conference on Decision and Control}, 2008,
  pp. 4351--4356.

\bibitem{cardinal2009poset}
J.~Cardinal, S.~Fiorini, G.~Joret, R.~Jungers, and J.~Munro, ``{An efficient
  algorithm for partial order production},'' in \emph{Proc. 41st annual ACM
  symposium on theory of computing}, 2009, pp. 93--100.

\bibitem{rotkowitz_acc06}
M.~Rotkowitz, ``Information structures preserved under nonlinear time-varying
  feedback,'' in \emph{Proc. American Control Conference}, 2006, pp.
  4207--4212.

\bibitem{cvx_toolbox}
M.~Grant and S.~Boyd, ``{CVX}: Matlab software for disciplined convex
  programming, version 1.21,'' \url{http://cvxr.com/cvx}, Aug. 2010.

\bibitem{donoho2006most}
D.~Donoho, ``{For most large underdetermined systems of equations, the minimal
  $\ell^1$-norm near-solution approximates the sparsest near-solution},''
  \emph{Communications on Pure and Applied Mathematics}, vol.~59, no.~7, pp.
  907--934, 2006.

\end{thebibliography}


\end{document}